\documentclass[a4paper, 12pt]{amsart}%
\usepackage{amsmath}
\usepackage{graphicx}%
\usepackage{amsfonts}%
\usepackage{amssymb}
\newtheorem{theorem}{Theorem}[section]

\theoremstyle{definition}

\begin{document}
\date{2017-5-29}
\title[Spectral properties of noncommutative Ricci flow]{Analyticity and spectral properties of noncommutative Ricci flow in a matrix geometry}
\author{Rocco Duvenhage, Wernd van Staden and Jan Wuzyk}
\address{Department of Physics\\
University of Pretoria\\
Pretoria 0002\\
South Africa}
\email{rocco.duvenhage@up.ac.za}

\begin{abstract}
We study\ a first variation formula for the eigenvalues of the Laplacian
evolving under the Ricci flow in a simple example of a noncommutative matrix
geometry, namely a finite dimensional representation of a noncommutative
torus. In order to do so, we first show that the Ricci flow in this matrix
geometry is analytic.

\end{abstract}
\maketitle

\section{Introduction}

In \cite{D}, Ricci flow was defined and studied in a simple example of a
matrix geometry, namely in a finite dimensional representation of a
noncommutative (or quantum) 2-torus. This was motivated by \cite{BM}, which
attempts to define Ricci flow in the usual infinite dimensional representation
of the noncommutative 2-torus by using a first variation formula for the
eigenvalues of the Laplace-Beltrami operator obtained in the classical case in
\cite{dC}.

In \cite{D}, however, the Ricci flow was defined more directly by a
noncommutative version of the Ricci flow equation, with no reference to the
spectrum of the Laplace-Beltrami operator or a first variation formula. In
this paper the aim is to show that a first variation formula can in fact also
be obtained for the Ricci flow as defined in \cite{D}.

The formula is obtained in Section \ref{var}. This is actually the second of
the two main results of the paper. In order to prove it, we first need to show
that the Ricci flow in \cite{D} is analytic, which is the first of our main
results, and is obtained in Section \ref{An}.

Section \ref{Ricci} briefly reviews the noncommutative Ricci flow from
\cite{D} in preparation for Section \ref{An}. The central object in Section
\ref{var} is the noncommutative Laplace-Beltrami operator. Section
\ref{Lap-Bel} presents this operator in the finite dimensional representation
in analogy to the known Laplace-Beltrami operator in the infinite dimensional
representation of the noncommutative torus.

The results of this paper contribute to showing that various properties of
Ricci flow in classical (i.e. commutative) differential geometry can be
systematically extended to a noncommutative example, indicating that Ricci
flow can be sensibly studied in the noncommutative case. Secondly, the paper
to some extent clarifies the similarities and differences between the
approaches taken in \cite{D} and \cite{BM} respectively.

As Ricci flow is of importance in differential geometry and related areas, it
seems plausible that extensions of results on Ricci flow to the noncommutative
case can ultimately be of value in noncommutative geometry and its
applications. Keep in mind that Ricci flow originated as part of Hamilton's
programme to prove the Poincar\'{e} conjecture \cite{H}, and that this
programme was indeed later completed by Perelman in \cite{Pe1, Pe2, Pe3}. In
Friedan's work at about the same time as \cite{H}, Ricci flow essentially also
appeared as part of a low order approximation to the renormalization group
equation of nonlinear sigma models in physics; see \cite{Fr1} for the initial
paper, but for a clearer formulation see in particular \cite[Section
II.1]{Fr2}. These remarks clearly illustrate the power and range of
applications of classical Ricci flow. Refer to \cite[Section 6]{BM} for a
brief discussion of the possibility of corresponding applications in the
noncommutative case.

Since in the formulation used in \cite{D} the usual partial differential
equation in fact becomes a system of ordinary differential equations, one can
use tools from linear algebra, systems of ordinary differential equations
(including the case on a complex domain, rather than just on a real interval),
complex analysis and perturbation theory of linear operators, to obtain
results that in the infinite dimensional representation are far more difficult
to prove or are, as yet, not accessible. An example of this is the convergence
of the Ricci flow to the flat metric, shown in \cite{D} using techniques from
systems of ordinary differential equations, but which had not yet been
obtained in \cite{BM}. In this paper we can use these techniques to derive the
analyticity of the Ricci flow, and consequently of the eigenvalues and
eigenvectors. Even in the classical case in \cite{dC}, on the other hand, the
existence of sufficiently smoothly parametrized eigenvalues and eigenvectors
had to be assumed in order to obtain the first variation formula. The setting
in \cite{D} is also very concrete and should for example be amenable to
numerical methods.

One should keep in mind though that the finite dimensional representation is
in many ways far simpler than the infinite dimensional case, so we would not
expect all these methods to extend easily to solve the corresponding problems
in the latter case. Nevertheless, one can learn a lot about noncommutative
geometry by studying simple examples, in particular in this case about
noncommutative partial differential equations. It seems plausible that some of
the ideas and techniques could be extended to more general situations, in
either infinite dimensional representations or other matrix geometries. This
could include noncommutative versions of partial differential equations other
than the Ricci flow. See for example the noncommutative heat equations
subsequently studied in \cite{Li} and \cite{L2}, in the same matrix geometry
as in \cite{D}. Also, see \cite{Ro} for partial differential equations in the
infinite dimensional representation of a noncommutative torus.

To conclude this introduction, we mention that the basic ideas regarding
matrix geometries originated in \cite{Hop} and \cite{M}. Noncommutative
geometry is, however, much broader (and older) than just the matrix case; see
for example \cite{C}. It remains an active field with much work still to be
done. In particular, Ricci flow has not yet been extensively studied in the
noncommutative case. Aside from the two papers on noncommutative Ricci flow
mentioned above, we can only point to \cite{CM} and \cite{V}.

\section{The Ricci flow\label{Ricci}}

Here we review the definition of Ricci flow as studied in \cite{D}, including
notation to be used throughout the paper.

Note that in the classical case, Ricci flow is given by \cite{H}
\[
\frac{\partial g_{\mu\nu}}{\partial t}=-2R_{\mu\nu},
\]
where $g_{\mu\nu}$ is a metric on a differentiable manifold, $R_{\mu\nu}$ is
the Ricci tensor, and $t$ is a real variable (``time''). Restricting ourselves
to surfaces and to metrics of the form $g_{\mu\nu}=c\delta_{\mu\nu}$, where
$\delta_{\mu\nu}$ is the Kronecker delta and $c$ is some strictly positive
function on the surface (a conformal rescaling factor), it can be shown that
this equation becomes
\[
\frac{\partial c}{\partial t}=(\partial_{1}\partial_{1}+\partial_{2}%
\partial_{2})\log c,
\]
where $\partial_{1}$ and $\partial_{2}$ are the partial derivatives with
respect to the coordinates $x^{1}$ and $x^{2}$ on the surface. This paper
studies a noncommutative version of the latter equation.

First, we recall the matrix geometry we are going to work with. We consider
two unitary $n\times n$ matrices $u$ and $v$ generating the algebra $M_{n}$ of
$n\times n$ complex matrices, and satisfying
\[
vu=quv,
\]
where
\[
q=e^{2\pi im/n}%
\]
for an $m\in\left\{  1,2,...,n-1\right\}  $ such that $m$ and $n$ are
relatively prime. Note that $q^{n}=1$, but $q^{j}\neq1$ for $j=1,...,n-1$. We
can for example use
\[
u=\left[
\begin{array}
[c]{cccc}%
1 &  &  & \\
& q &  & \\
&  & \ddots & \\
&  &  & q^{n-1}%
\end{array}
\right]
\]
and
\[
v=\left[
\begin{array}
[c]{cccc}%
0 & 1 &  & \\
& 0 & \ddots & \\
&  & \ddots & 1\\
1 &  &  & 0
\end{array}
\right]  ,
\]
where the blank spaces are filled with zeroes. It is straightforward to check
that the commutant of the set $\{u,v\}$ only consists of scalar multiples of
the identity matrix $I$, so $u$ and $v$ indeed generate $M_{n}$. These
matrices appeared in physics at least as far back as \cite{S}, in relation to
quantum mechanics, but the geometric interpretation in terms of a
noncommutative (or ``fuzzy'') torus seems to have only come later; see for
example \cite[Section 2]{M}.

For any two Hermitian matrices $x$ and $y$ such that
\[
u=e^{\frac{2\pi i}{n}x}%
\]
and
\[
v=e^{\frac{2\pi i}{n}y},
\]
we define derivations $\delta_{1}$ and $\delta_{2}$ on $M_{n}$ by the
commutators
\[
\delta_{1}:=\left[  y,\cdot\right]
\]
and
\[
\delta_{2}:=-\left[  x,\cdot\right]  ,
\]
which are analogues of the partial derivatives $\frac{1}{i}\partial_{\mu}$ in
the classical case (please refer to \cite[Section 2]{D} for a discussion of
this). Note that such $x$ and $y$ exist but are not uniquely determined by $u$
and $v$. However, our analysis will not depend on the choices made. The
existence of $x$ is clear for the diagonal $u$ above. For $v$ above, we first
diagonalize using the Fourier transform for the group $\mathbb{Z}_{n}$, obtain
a $y$ in this basis, and then transform back to the original basis.

Note that $M_{n}$ is an involutive algebra, i.e. a $\ast$-algebra, with the
involution given by the usual Hermitian adjoint of a matrix $a$, denoted by
$a^{\ast}$. In the operator norm, $M_{n}$ is a unital C*-algebra. Although the
theory of C*-algebras will not appear in this paper, the C*-algebraic point of
view helps to connect what we do in this paper with the setting used in
\cite{BM} and \cite{CT}. Furthermore, we for the most part only need abstract
properties of the derivations (like those listed in \cite[Proposition 2.1]%
{D}), rather than the explicit definitions of $u$ and $v$, and $\delta_{1}$
and $\delta_{2}$, given above. The exceptions are one point in each of the
proofs of Theorems \ref{RA} and \ref{varSt}, where we do use the fact that the
derivations are given by commutators with some fixed matrices. To emphasize
this abstract point of view, we usually rather use the notation
\[
A=M_{n}%
\]
for our unital $\ast$-algebra. The unit of $A$ is the $n\times n$ identity
matrix $I$.

Using the derivations above, we define a noncommutative analogue of a
Laplacian as an operator on $A$:
\begin{equation}
\triangle:=\delta_{1}^{2}+\delta_{2}^{2}, \label{Lap}%
\end{equation}
i.e. $\triangle a=[y,[y,a]]+[x,[x,a]]$ for all $a\in A$.

The Hilbert-Schmidt inner product
\[
\left\langle a,b\right\rangle :=\tau\left(  a^{\ast}b\right)
\]
on $A$ then becomes relevant. Here $\tau$ denotes the usual trace on $M_{n}$,
i.e. the sum of the diagonal entries of a matrix. One can interpret
$\tau\left(  a\right)  $ as a noncommutative integral of the complex-valued
``function'' $a$, corresponding in the classical case to the integral over the
flat torus.

We denote the Hilbert space $(A,\left\langle \cdot,\cdot\right\rangle )$ by
$H$.

Note that the Laplacian
\[
\triangle:H\rightarrow H
\]
is a positive operator (see \cite[Proposition 2.1]{D}), so it corresponds to
$-(\partial_{1}\partial_{1}+\partial_{2}\partial_{2})$ in the classical case.

Now we turn to noncommutative metrics.

For $a\in A$, we write
\[
a>0
\]
if $a$ is a positive operator, i.e. if it can be written as $a=b^{\ast}b$ for
some $b\in A$, and in addition $0$ is not an eigenvalue of $a$. In other
words, $a>0$ means that $a$ is a Hermitian $n\times n$ matrix whose
eigenvalues are strictly positive. We also write
\[
P=\{a\in A:a>0\}
\]
to denote the set of all these elements.

A \emph{noncommutative metric} is any $c\in P$. This is the noncommutative
version of the metric $c\delta_{\mu\nu}$ in the classical case above. Given
any $c\in P$, we also consider the Hilbert space $H_{c}$ given by the inner
product%
\begin{equation}
\left\langle a,b\right\rangle _{c}:=\varphi\left(  a^{\ast}b\right)
\label{KrInp}%
\end{equation}
on the vector space $A$, where
\begin{equation}
\varphi(a):=\tau(ca) \label{fi}%
\end{equation}
for all $a\in A$. The positive linear functional $\varphi$ is a noncommutative
version of the integral over a curved surface (with metric given by a
conformal rescaling factor as above) in the classical case, a point which
becomes relevant in Section \ref{var}.

This setup for noncommutative metrics is adapted from the infinite dimensional
representation as studied in \cite{CT}. The Hilbert space $H_{c}$ will play a
central role in Sections \ref{Lap-Bel} and \ref{var}, where we work with the
noncommutative version of the Laplace-Beltrami operator, which is essentially
a Laplacian on a curved noncommutative space.

A noncommutative version of the classical Ricci flow equation above is the
following:
\begin{equation}
\frac{d}{dt}c(t)=-\triangle\log c(t), \label{R}%
\end{equation}
where $c(t)\in P$ denotes the metric at time $t$. Here we define $\log a$ of
an $a\in P$ by diagonalizing $a$, applying $\log$ to each of the diagonal
entries, and then returning to the original basis (this is the Borel
functional calculus in finite dimensions). In this sense we can view $\log$ as
the real logarithm applied to elements of $P$. Below, when studying the
analyticity of this Ricci flow, we also use the principal complex logarithm on
$D:=\mathbb{C}\backslash(-\infty,0]$, denoted by $\operatorname{Log}$, and
applied to the larger set of matrices whose spectra are in $D$.

\section{Analyticity of the Ricci flow\label{An}}

In \cite[Section 3]{D}, we used the theory of systems of differential
equations to show that, given initial conditions, the Ricci flow equation
(\ref{R}) has a unique solution $c$ and that this solution is $C^{1}$, i.e. it
is differentiable and its derivative is continuous. Here we go further and
obtain the first of the main results of this paper, namely that the Ricci flow
is in fact analytic. We use the theory of systems of differential equations on
a complex domain to do so. We collect these results as follows:

\begin{theorem}
\label{RA}Let $c_{t_{0}}\in P$ be any initial noncommutative metric at the
initial time $t_{0}\in\mathbb{R}$. Then the noncommutative Ricci flow, Eq.
(\ref{R}), has a unique $C^{1}$ solution $c$ on any interval $[t_{0},t_{1}]$
with $t_{1}\geq t_{0}$, and also on the interval $[t_{0},\infty)$, such that
$c(t_{0})=c_{t_{0}}$. In addition, such a solution is necessarily analytic,
i.e. at each point $t_{2}\in\lbrack t_{0},\infty)$ there is a number
$\varepsilon>0$ such that each entry in the matrix $c(t)\in P$ is a power
series in $t-t_{2}$ for all $t$ in the interval $[t_{0},\infty)\cap
(t_{2}-\varepsilon,t_{2}+\varepsilon)$.
\end{theorem}

\begin{proof}
We first consider the $C^{1}$ property globally, and afterwards we study
analyticity locally. Although the $C^{1}$ property was considered in \cite{D},
we again look at certain aspects carefully and in more detail here, since the
results we obtain on the way are subsequently used in proving analyticity.

We start by looking at the properties of $\operatorname{Log}$ when applied to
matrices, and viewed as a function of several complex variables. Below we
define $\operatorname{Log}a$ for all $a$ in the open set $B:=\{a\in A:$
$\sigma(a)\subset D\}$, where $\sigma(a)$ is the spectrum of $a$, and
$D=\mathbb{C}\backslash(-\infty,0]$. That $B$ is indeed open in $A$, follows
from fact that the eigenvalues of a matrix depend continuously on the matrix, 
since the roots of a polynomial depend continuously on the
coefficients of the polynomial (see for example \cite[Section I.1]{Mar}). By
using the analytic functional calculus, we set%
\begin{equation}
\operatorname{Log}a:=\frac{1}{2\pi i}\int_{\Gamma}(zI-a)^{-1}%
\operatorname{Log}zdz \label{Log}%
\end{equation}
for any positively oriented simple closed smooth contour $\Gamma$ in $D$,
surrounding $\sigma(a)$, for all $a\in B$. This definition is independent of
$\Gamma$. See for example \cite[Section III.3]{A} for further details. Note
that because of the Cauchy integral formula, for $a>0$ this definition
corresponds to the definition for $\log a$ discussed at the end of the
previous section.

The inverse $\cdot^{-1}:\operatorname{Inv}(A)\rightarrow\operatorname{Inv}%
(A):a\mapsto a^{-1}$, on the open set $\operatorname{Inv}(A)$ of invertible
matrices, is differentiable in each entry $a_{jk}\in\mathbb{C}$ of the matrix
$a$ being inverted. One can see this from the formula for the inverse obtained
from Cramer's rule. So each entry in $a^{-1}$ is an analytic complex function
of each entry of $a$ separately. Therefore, by the Hartogs theorem (see for
example \cite[Theorem 2.2.8]{Ho}), each entry in the matrix $a^{-1}$ is an
analytic function in several complex variables, the variables being the
entries of the matrix $a$. The derivatives of these analytic functions are
therefore themselves analytic, and hence continuous functions. Because of this
we can differentiate Eq. (\ref{Log}) under the integral with respect to the
entries of $a$ (refer for example to \cite[Theorem VIII.6.A3]{La}), to see
that each entry in the matrix $\operatorname{Log}a$ is (again by the Hartogs
theorem) an analytic function in several complex variables. Since $\triangle$
is given by commutators with certain matrices, we immediately also know that
the entries of the matrix $\triangle\operatorname{Log}a$ are each analytic
functions in several complex variables, the variables still being the entries
of $a$.

In particular, $a\mapsto\triangle\log a$ is a $C^{1}$ function on $P$, which
means (by the theory of systems of ordinary differential equations; see for
example \cite[Section 1.1]{Ch}) that the initial value problem for Eq.
(\ref{R}) has a unique $C^{1}$ solution on any interval for which the solution
stays in $P$. This has already been discussed in \cite[Section 3]{D}, where in
particular it was shown that the solution remains in $P$ for all $t\geq t_{0}%
$, i.e. it exists (and is necessarily unique) on $[t_{0},\infty)$.

However, we are now interested in the analyticity of this solution $c$. We
approach this problem using the theory of systems of ordinary differential
equations on a complex domain.

For any $t_{2}\geq t_{0}$, consider the system of differential equations given
in matrix form by%
\[
\frac{d}{dz}w(z)=-\triangle\operatorname{Log}w(z)
\]
where $w$ is required to be a function on some neighbourhood of $t_{2}$ in
$\mathbb{C}$, with values in the open set $B$ consisting of matrices $a\in A$
such that $\sigma(a)$ lies in $D$. I.e. the values of $w$ should be in the
domain of $\operatorname{Log}$ viewed as a matrix function as defined above.

Because of the analyticity of $a\mapsto\triangle\operatorname{Log}a$ shown
above, and the fact that it is consequently locally Lipschitz (see for example
\cite[Section 6.3]{La1}), it follows by \cite[Theorem 2.2.2]{Hi} that, on a
small enough open disc in $\mathbb{C}$ of radius $\varepsilon$ around $t_{2}$,
this system has a unique analytic solution such that $w(t_{2})=c(t_{2})$. I.e.
on this disc the entries of $w(z)$ are analytic functions, and therefore have
power series expansions on this disc.

Restricting such a solution to real elements $[t_{0},\infty)\cap
(t_{2}-\varepsilon,t_{2}+\varepsilon)$ in the disc, we necessarily obtain our
solution $c$ with values in $P$ discussed above on this interval around
$t_{2}$. The reason for this is that the system
\[
\frac{d}{dt}w(t)=-\triangle\operatorname{Log}w(t)
\]
for real $t$, has a unique solution by the theory of systems of ordinary
differential equations on a real interval, just as for the case of $c$ above.
So given the condition $w(t_{2})=c(t_{2})$, such a solution of $w$ must in
fact be the solution we already have, namely $c$, on the interval in question.
But the restriction of the solution on a complex domain to $[t_{0},\infty
)\cap(t_{2}-\varepsilon,t_{2}+\varepsilon)$ mentioned above is exactly such a
solution, hence on this interval it is indeed $c$.

This means that the entries of $c$ have power series expansions at each
$t_{2}$, and therefore $c$ is analytic.
\end{proof}

In addition, \cite[Theorem 3.2]{D} also showed the convergence of $c$ to the
flat metric (proportional to the identity matrix $I$) as $t$ goes to infinity,
as well as monotonicity of the determinant and preservation of the trace under
the flow. However, analyticity of the Ricci flow is the property of
fundamental importance in obtaining a first variation formula for the
eigenvalues of the Laplace-Beltrami operator in Section \ref{var}.

\section{The Laplace-Beltrami operator\label{Lap-Bel}}

In the classical case, one can define a Laplacian that incorporates the
metric, the so-called Laplace-Beltrami operator. For the classical metric
$g_{\mu\nu}=c\delta_{\mu\nu}$ mentioned in Section \ref{Ricci}, the
Laplace-Beltrami operator is of the form%
\[
-\frac{1}{c}(\partial_{1}\partial_{1}+\partial_{2}\partial_{2}),
\]
if we use the convention that it should be a positive operator. (Often the
minus sign is dropped, then minus the Laplace-Beltrami operator would be a
positive operator.) Note that it reduces to the Laplacian $-(\partial
_{1}\partial_{1}+\partial_{2}\partial_{2})$ for the flat metric $c=I$.

We need to write down a suitable noncommutative version of this operator,
which should have a similar form and similar properties. In particular, it
should also be a positive operator, and it should reduce to the noncommutative
Laplacian $\triangle$ for the flat metric $c=I$. A natural choice is the
operator
\begin{equation}
\triangle_{c}:H_{c}\rightarrow H_{c}:a\mapsto(\triangle a)c^{-1} \label{LB}%
\end{equation}
for any metric $c\in P$, where the product of $\triangle a$ and $c^{-1}$ is
taken in $A$. Keep in mind from Section \ref{Ricci} that $H_{c}%
=(A,\left\langle \cdot,\cdot\right\rangle _{c})$. Note that the operator
$\triangle_{c}:H_{c}\rightarrow H_{c}$ is indeed positive, since%
\[
\left\langle a,\triangle_{c}a\right\rangle _{c}=\tau(ca^{\ast}(\triangle
a)c^{-1})=\tau(a^{\ast}\triangle a)=\left\langle a,\triangle a\right\rangle
\geq0
\]
by the fact that $\triangle:H\rightarrow H$ is positive (see \cite[Proposition
2.1]{D}), where $H=H_{I}$ as in Section \ref{Ricci}. The operator
$\triangle_{c}$ also clearly reduces to $\triangle$ when $c=I$. We therefore
use (\ref{LB}) as our definition of the Laplace-Beltrami operator.

Note that the alternative definition $c^{-1}\triangle a$ for $\triangle_{c}a$,
which may at first seem to be the more obvious choice, would fail, since it
would not guarantee positivity of $\triangle_{c}$.

The right multiplication by $c^{-1}$ in the definition of the Laplace-Beltrami
operator also appears in the infinite dimensional representation of the
noncommutative 2-torus. See for example \cite[Remark 2.2]{FK}.

In the next section it will also be convenient to represent $\triangle_{c}$ on
$H$ instead of $H_{c}$. To do this we define a unitary operator%
\[
U_{c}:H_{c}\rightarrow H:a\mapsto ac^{1/2}%
\]
where the product $ac^{1/2}$ is taken in the algebra $A$. Note that this is
indeed unitary, since%
\[
\left\langle U_{c}a,U_{c}b\right\rangle =\tau((ac^{1/2})^{\ast}bc^{1/2}%
)=\tau(ca^{\ast}b)=\left\langle a,b\right\rangle _{c}%
\]
for all $a,b\in H_{c}$, and $U_{c}$ is invertible, since $c^{1/2}$ is. We can
therefore represent $\triangle_{c}$ on $H$ by the positive operator%
\[
\bar{\triangle}_{c}:=U_{c}\triangle_{c}U_{c}^{\ast}:H\rightarrow H,
\]
for which we have%
\begin{equation}
\bar{\triangle}_{c}a=(\triangle(ac^{-1/2}))c^{-1/2} \label{LBV}%
\end{equation}
for all $a\in H$.

\section{The first variation formula\label{var}}

This section presents our second main result, namely a version in our context
of the classical first variation formula obtained in \cite[Corollary 2.3]{dC}
for the eigenvalues of the time-dependent Laplace-Beltrami operator given by
the classical Ricci flow. The analyticity of the noncommutative Ricci flow,
given by Theorem \ref{RA}, will be used in proving this result. We then
discuss this first variation formula in relation to \cite{BM} and the
classical case.

Keep in mind from Section \ref{Ricci} that $\tau$ is the trace on $M_{n}$,
that $\triangle$ is the flat Laplacian given by Eq. (\ref{Lap}), and that the
Hilbert space $H_{c}$ is defined by the inner product in Eq. (\ref{KrInp}). To
make some expressions easier to read, we denote $c(t)$ also by%
\[
c_{t}:=c(t).
\]
In terms of this notation, we can formulate our second main result as follows:

\begin{theorem}
\label{varSt}Let $c$ be the Ricci flow on $[t_{0},\infty)$ as given by Theorem
\ref{RA}. Then the eigenvalues and normalized eigenvectors of $\triangle
_{c_{t}}$ can be obtained as analytic functions of $t$, and for each such
eigenvalue $\lambda_{t}$ and a corresponding normalized eigenvector $a_{t}\in
H_{c_{t}}$, we have the first variation formula
\[
\frac{d\lambda_{t}}{dt}=\lambda_{t}\tau(|a_{t}|^{2}\triangle\log c_{t}),
\]
where $|a_{t}|^{2}:=a_{t}^{\ast}a_{t}$, for all $t\in\lbrack t_{0},\infty)$.
(At $t=t_{0}$, this derivative can be viewed as the right-hand derivative.)
\end{theorem}

\begin{proof}
We are going to apply perturbation theory of linear operators to
$\triangle_{c(t)}$. In order to do so, we work via $\bar{\triangle}%
_{c(t)}:H\rightarrow H$ as defined in the previous section, since then we have
an operator on the same space $H$ for all $t$. In order to apply perturbation
theory (see \cite[Chap. VII, Sections 2 and 3]{K}), we show that $t\mapsto
\bar{\triangle}_{c(t)}$ can be extended to a neighbourhood of $[t_{0},\infty)$
in $\mathbb{C}$, i.e. to an open set in $\mathbb{C}$ containing $[t_{0}%
,\infty)$. More precisely, we need the extension of $t\mapsto\bar{\triangle
}_{c(t)}a$ to be analytic for all $a\in A$.

Because of Eq. (\ref{LBV}), we start by showing that $t\mapsto c(t)^{-1/2}$ is
analytic: $t\mapsto c(t)$ is analytic by Theorem \ref{RA}, and therefore at
each point $t_{1}$ in $[t_{0},\infty)$ has a power series expansion for each
of its entries which can be used to extend these entries to analytic complex
functions on a neighbourhood (a disc) $N$ of $t_{1}$ in $\mathbb{C}$, giving
us a complex matrix valued function $z\mapsto w(z)$ extending $t\mapsto c(t)$
to $N$. Secondly, the entries of $a^{-1}$ are complex analytic functions of
several complex variables (the entries of $a$), for $a$ in the set of
invertible elements Inv$(A)$ of $A$, as mentioned in the proof of Theorem
\ref{RA}. Thirdly, the entries of the square root $a^{1/2}$, for $a\in A$
whose spectrum $\sigma(a)$ lies in $D=\mathbb{C}\backslash(-\infty,0]$, are
complex analytic functions of several complex variables (the entries of $a$),
using the same argument as for $\operatorname{Log}a$ in the proof of Theorem
\ref{RA}. (Here we are using the branch of the square root given by
$e^{\frac{1}{2}\operatorname{Log}}$, and we express $a^{1/2}$ in terms of the
analytic functional calculus.) Choosing the radius of $N$ small enough, we can
ensure that $\sigma(w(z))\subset D$, so $w(z)^{-1}$ exists and $\sigma
(w(z)^{-1})\subset D$, for $z\in N$. This is because $\sigma(c(t_{1}%
))\subset(0,\infty)\subset D$, and for $N$ small enough, each element of
$\sigma(w(z))$ will be as close to some element of $\sigma(c(t_{1}))$ as we
require, simply because $w$ is continuous and $w(t_{1})=c(t_{1})$. Here we
have again used the fact that the eigenvalues of a matrix depend continuously
on the entries of the matrix, as mentioned in the proof of Theorem \ref{RA}.
Hence the composition of the above mentioned three functions, namely
$z\mapsto(w(z)^{-1})^{1/2}=w(z)^{-1/2}$ on $N$, is well-defined, and its
entries are complex differentiable (because of the above mentioned analyticity
of the three functions) and therefore analytic.

It follows that the entries of $z\mapsto\bar{\triangle}_{w(z)}a:=(\triangle
(aw(z)^{-1/2}))w(z)^{-1/2}$ are differentiable with respect to $z$, and
therefore analytic on $N$. Thus its restriction $t\mapsto\bar{\triangle
}_{c(t)}a$ to the real line is also analytic.

Because we are working in finite dimensions, the fact that $t\mapsto
\bar{\triangle}_{c(t)}a$ is analytic for all $a$, is equivalent to the entries
of $\bar{\triangle}_{c(t)}$ (viewed as an $n^{2}\times n^{2}$ matrix acting on
an $n^{2}$ dimensional space) being analytic. One can now analytically extend
it to a neighbourhood of $[t_{0},\infty)$. To do this we consider power series
expansions for the entries of the $n^{2}\times n^{2}$ matrix $\bar{\triangle
}_{c(t)}$ on a neighbourhood in $\mathbb{R}$ of each $t_{1}\in\lbrack
t_{0},\infty)$, allowing us to define an analytic extension $z\mapsto T(z)$ of
$t\mapsto\bar{\triangle}_{c(t)}$ to a disc in $\mathbb{C}$ around $t_{1}$, for
each $t_{1}\in\lbrack t_{0},\infty)$. Since the extensions on two overlapping
discs are equal on a non-empty open interval in the real line, they are equal
on the overlap because of analyticity. Hence we have an analytic extension
$z\mapsto T(z)$ of $t\mapsto\bar{\triangle}_{c(t)}$ to a neighbourhood of
$[t_{0},\infty)$, namely to the union $D_{0}$ of all these discs. Note that
$D_{0}$ is symmetric around the real axis.

To apply the results from \cite[Chap. VII, Sections 2 and 3]{K}, we
furthermore need $T(z)^{\ast}=T(\bar{z})$ to hold. To see that this is indeed
true, represent $H$ as $\mathbb{C}^{n^{2}}$ with its usual inner product, in
other words we choose some orthonormal basis in $H$. Then we can represent
$T(z)$ as an $n^{2}\times n^{2}$ matrix, such that $T(z)^{\ast}$ is simply the
usual Hermitian adjoint of $T(z)$, i.e. transpose and entrywise complex
conjugation. By analyticity, at each $t_{1}\in\lbrack t_{0},\infty)$ each
entry of $T(z)$ is a power series in $z-t_{1}$ for all $z$ in some disc $N$
centered at $t_{1}$, with radius only depending on $t_{1}$, since we can use
the smallest radius that still works for all entries. Say the power series for
the $(k,l)$ entry of $T(z)$ is given by
\[
T(z)_{kl}=\sum_{j=0}^{\infty}m_{jkl}(z-t_{1})^{j}%
\]
for all $z\in N$ and all $(k,l)$. Then in particular the $(k,l)$ entry of
$\bar{\triangle}_{c(t)}$ in the same representation is $(\bar{\triangle
}_{c(t)})_{kl}=\sum_{j=0}^{\infty}m_{jkl}(t-t_{1})^{j}$ for $t\in N\cap\lbrack
t_{0},\infty)$. But $\bar{\triangle}_{c(t)}$ is self-adjoint, since it is a
positive operator because of its definition in the previous section. It
follows that
\[
\bar{m}_{jlk}=m_{jkl},
\]
from which it in turn follows that $T(z)^{\ast}=T(\bar{z})$, as required.

By the perturbation theory of linear operators, in particular \cite[Chap. VII,
Theorem 3.9]{K}, we now conclude that the eigenvalues and normalized
eigenvectors of $\bar{\triangle}_{c(t)}:H\rightarrow H$ can be parametrized as
analytic functions of $t\in\lbrack t_{0},\infty)$. (Also see \cite[Theorem
(A)]{KMR} for a review, and \cite{Re} for the original literature.) Given such
an eigenvalue $\lambda_{t}$ and a corresponding normalized eigenvector
$\bar{a}_{t}$, it follows that $\lambda_{t}$ is also an eigenvalue of
$\triangle_{c(t)}$, and that $a_{t}:=U_{c(t)}^{\ast}\bar{a}_{t}=\bar{a}_{t}%
c_{t}^{-1/2}\in H_{c(t)}$ is a corresponding normalized eigenvector. Since
both $\bar{a}_{t}$ and $c_{t}^{-1/2}$ are analytic functions of $t$, the same
is true of $a_{t}$. So, we have transformed back to the representation on
$H_{c(t)}$, to see that the eigenvalues and normalized eigenvectors of
$\triangle_{c(t)}:H_{c(t)}\rightarrow H_{c(t)}$ can be parametrized as
analytic functions of $t\in\lbrack t_{0},\infty)$. (In \cite[Section 1]{dC} a
result of this form was assumed without proof.)

Ignoring time dependence for the moment, consider a metric $c\in P$, and let
$\lambda$ be any eigenvalue of $\triangle_{c}$, with $a\in H_{c}$ a
corresponding, but not necessarily normalized, eigenvector, i.e.
$\triangle_{c}a=\lambda a$, then we have
\[
a^{\ast}(\triangle a)c^{-1}=\lambda a^{\ast}a
\]
from which
\[
\lambda=\frac{\tau(a^{\ast}\triangle a)}{\tau(ca^{\ast}a)}%
\]
follows. Here $\tau(ca^{\ast}a)=\left\langle a,a\right\rangle _{c}=\left\|
a\right\|  _{c}^{2}\neq0$ is the norm squared of the eigenvector $a$ in the
Hilbert space $H_{c}$. (This step in our proof is closely related to the
approach taken in \cite{dC}. The rest of our proof, however, is rather
different from the proof of the classical first variation formula given in
\cite{dC}.)

Now consider any eigenvalue $\lambda_{t}$ and corresponding eigenvector
$a_{t}$ of $\triangle_{c(t)}$, both analytic in $t$, but with the eigenvector
not necessarily normalized. Denoting time derivatives by $\dot{\lambda}%
_{t}=d\lambda_{t}/dt$, and similarly for other time-dependent objects, we
obtain%
\begin{equation}
\dot{\lambda}_{t}=\frac{\tau(\dot{a}_{t}^{\ast}\triangle a_{t}+a_{t}^{\ast
}\triangle\dot{a}_{t})}{\tau(c_{t}a_{t}^{\ast}a_{t})}-\frac{\tau(a_{t}^{\ast
}\triangle a_{t})}{\tau(c_{t}a_{t}^{\ast}a_{t})^{2}}\tau(\dot{c}_{t}%
a_{t}^{\ast}a_{t}+c_{t}\dot{a}_{t}^{\ast}a_{t}+c_{t}a_{t}^{\ast}\dot{a}_{t})
\label{tafg}%
\end{equation}
where we have used the fact that $\triangle$ is defined in terms of
commutators with some fixed operators, and therefore the time derivative can
be taken over $\triangle$.

Let us now specialize to the case of a normalized eigenvector $a_{t}$, as
obtained above. Since Eq. (\ref{tafg}) holds for an eigenvector in general, it
in particular holds when it is normalized, i.e. when $\tau(c_{t}a_{t}^{\ast
}a_{t})=1$. So, using the fact that $\triangle:H\rightarrow H$ is Hermitian,
as well as the eigenvalue equation for $\triangle_{c(t)}$ in the form
$\triangle a_{t}=\lambda_{t}a_{t}c_{t}$, and lastly the Ricci flow equation,
namely Eq. (\ref{R}), we obtain%
\begin{align*}
\dot{\lambda}_{t}  &  =\tau(\dot{a}_{t}^{\ast}\triangle a_{t}+a_{t}^{\ast
}\triangle\dot{a}_{t})-\tau(a_{t}^{\ast}\triangle a_{t})\tau(\dot{c}_{t}%
a_{t}^{\ast}a_{t}+c_{t}\dot{a}_{t}^{\ast}a_{t}+c_{t}a_{t}^{\ast}\dot{a}_{t})\\
&  =\tau(\dot{a}_{t}^{\ast}\triangle a_{t}+(\triangle a_{t})^{\ast}\dot{a}%
_{t})-\lambda_{t}\tau(\dot{c}_{t}a_{t}^{\ast}a_{t}+c_{t}\dot{a}_{t}^{\ast
}a_{t}+c_{t}a_{t}^{\ast}\dot{a}_{t})\\
&  =\lambda_{t}\tau(\dot{a}_{t}^{\ast}a_{t}c_{t}+c_{t}a_{t}{}^{\ast}\dot
{a}_{t})-\lambda_{t}\tau(\dot{c}_{t}a_{t}^{\ast}a_{t}+c_{t}\dot{a}_{t}^{\ast
}a_{t}+c_{t}a_{t}^{\ast}\dot{a}_{t})\\
&  =-\lambda_{t}\tau(\dot{c}_{t}a_{t}^{\ast}a_{t})\\
&  =\lambda_{t}\tau(a_{t}^{\ast}a_{t}\triangle\log c_{t}),
\end{align*}
as required.
\end{proof}

We note that this formula can also be written as
\begin{equation}
\frac{d\lambda_{t}}{dt}=\lambda_{t}\varphi_{t}(|a_{t}|^{2}\triangle_{c_{t}%
}\log c_{t}), \label{varf}%
\end{equation}
where $\varphi_{t}(a):=\tau(c_{t}a)$ for all $a\in A$. The significance of
this is that $\triangle_{c_{t}}\log c_{t}=(\triangle\log c_{t})c_{t}^{-1}$
corresponds exactly to the formula for scalar curvature in the classical case.
Therefore Eq. (\ref{varf}) appears to be a reasonable analogue of the
classical formula
\[
\frac{d\lambda_{t}}{dt}=\lambda_{t}\int f_{t}^{2}R_{t}d\mu_{t}%
\]
obtained in \cite[Corollary 2.3]{dC}, where $R_{t}$ is the classical scalar
curvature, $f_{t}$ is an eigenfunction of the classical Laplace-Beltrami
operator at time $t$, and the integral $\int(\cdot)d\mu_{t}$ over the surface
in question corresponds to the positive linear functional $\varphi_{t}$.
However, it should also be pointed out that $(\triangle\log c_{t})c_{t}^{-1}$
is not a sensible noncommutative scalar curvature. It can for example not even
be expected to be a Hermitian element of $A$. See for example \cite[Section
5]{D} for how one can define a more sensible noncommutative scalar curvature
in this context from the noncommutative Ricci flow. Therefore the analogy
between Eq. (\ref{varf}) and the classical case is not perfect, since Eq.
(\ref{varf}) is not in terms of scalar curvature, but rather in terms of a
noncommutative object having a form similar to the classical scalar curvature.

This result also indicates some similarity between our setting for Ricci flow,
and that of \cite{BM}, where the first variation formula is used as the basis
for the noncommutative Ricci flow. However, there a much more complicated
object is used in the place of the classical $R$; see in particular
\cite[Theorem 3.5]{BM}.

In \cite{dC} the existence of eigenvalues and eigenvectors which are
sufficiently smooth (namely $C^{1}$) in $t$ is assumed, rather than proved. It
should be added, though, that in that paper classical manifolds more general
than just surfaces are considered, so such assumptions may be unavoidable
there. In \cite[Section 1]{dC}, it is also mentioned that for the classical
Ricci flow one can not in general expect analyticity in $t$.

In \cite{BM} it appears that an analogous assumption of sufficient smoothness
is made in the case of the noncommutative torus (in the infinite dimensional
representation). It is however not clear how the time-dependent eigenvectors
in \cite{BM} arise in the first place, since there the Ricci flow is defined
and studied by expressing the time-derivative of an eigenvalue as a type of
first variation formula involving the corresponding eigenvector, without
further equations from which the time-dependence of the eigenvector can be obtained.

We now conclude with further general remarks on Theorem \ref{varSt} in
relation to the classical case:

In the classical case, one can work with real-valued eigenfunctions $f$, which
in the noncommutative case correspond to Hermitian elements of $A$. However,
noncommutativity appears to prohibit restricting ourselves to Hermitian
elements of $A$ as eigenvectors. The reason for this is that if we split a not
necessarily Hermitian eigenvector $a$ of $\triangle_{c}$ into its Hermitian
parts, i.e. $a=a_{1}+ia_{2}$, with $a_{1},a_{2}\in A$ Hermitian, the
noncommutativity makes it impossible to show that $a_{1}$ and $a_{2}$ (or
whichever of them are not zero) are eigenvectors, since $(\triangle
_{c}a)^{\ast}=c^{-1}\triangle a_{1}-ic^{-1}\triangle a_{2}$, which need not be
equal to $\triangle_{c}a_{1}-i\triangle_{c}a_{2}$. This is unlike the
classical case, where the real and imaginary parts (at least whichever of them
are not zero) of an eigenfunction are in fact themselves eigenfunctions, as is
easily verified. This is why we have to work with $|a_{t}|^{2}$ rather than
$a_{t}^{2}$ in the place of $f_{t}^{2}$.

In the proof of the classical first variation formula in \cite{dC}, the
eigenfunctions (other than $1$) of the Laplace-Beltrami operator for a given
metric average to zero, i.e.
\[
\int fd\mu=0,
\]
where $\int(\cdot)d\mu$ is the integral over the manifold. We did not
explicitly use the noncommutative version of this fact in the proof above, but
it is interesting to note that it is indeed true. To see this, consider any
eigenvalue $\lambda$ and corresponding eigenvector $a\in H_{c}$ of
$\triangle_{c}$, for any metric $c\in P$. Then $\triangle a=\lambda ac$, from
which (see Eq. (\ref{fi})) we have $\lambda\varphi(a)=\tau(\triangle a)=0$.
When $\lambda=0$, we necessarily have that $a$ is proportional to $I$, since
$\ker\triangle=\mathbb{C}I$ by \cite[Proposition 2.1]{D}, showing that%
\[
\ker\triangle_{c}=\mathbb{C}I
\]
as well. On the other hand, for all other eigenvectors $a$, i.e. those not
proportional to $I$, we have $\lambda\neq0$, so
\[
\varphi(a)=0
\]
in exact analogy to the classical case.

\section*{Acknowledgments}

This work was supported by the National Research Foundation of South Africa.

\end{document}